\documentclass[11pt]{article}
\usepackage{fullpage}

\usepackage{amsmath} % assumes amsmath package installed
\usepackage{amssymb}  % assumes amsmath package installed

\usepackage{amsthm}
\usepackage{color}

\usepackage{cite}

\usepackage{graphicx}
\usepackage{epstopdf}
\usepackage{subfigure}

\newtheorem{theorem}{Theorem}
\newtheorem{lemma}{Lemma}
\newtheorem{corollary}{Corollary}
\newtheorem{assumption}{Assumption}

\theoremstyle{remark}
\newtheorem{remark}{Remark}

\title{\LARGE \bf
A Distributed Stochastic Gradient Tracking Method
}

\author{Shi~Pu and Angelia~Nedi{\'c}% <-this % stops a space
\thanks{This work was supported in parts by the NSF grant CPS 15-44953 and the ONR grant no.\ N00014-12-1-0998.}% <-this % stops a space
\thanks{Shi Pu and Angelia Nedi{\'c} are with the  School of Electrical, Computer, and Energy Engineering, Arizona
	State University, Tempe, AZ 85287, USA.
        {\tt\small (emails: sp3dw@virginia.edu, Angelia.Nedich@asu.edu)}}%
}

\newcommand{\mx}{\mathbf{x}}
\newcommand{\my}{\mathbf{y}}

\newcommand{\ox}{\overline{x}}
\newcommand{\oy}{\overline{y}}

\newcommand{\bE}{\mathbb{E}}

\newcommand{\T}{\intercal}

\begin{document}

\maketitle
\thispagestyle{empty}
\pagestyle{empty}

%%%%%%%%%%%%%%%%%%%%%%%%%%%%%%%%%%%%%%%%%%%%%%%%%%%%%%%%%%%%%%%%%%%%%%%%%%%%%%%%
\begin{abstract}
In this paper, we study the problem of distributed multi-agent optimization over a network, where each agent possesses a local cost function that is smooth and strongly convex. The global objective is to find a common solution that minimizes the average of all cost functions. Assuming agents only have access to unbiased estimates of the gradients of their local cost functions, we consider a distributed stochastic gradient tracking method. We show that, in expectation, the iterates generated by each agent are attracted to a neighborhood of the optimal solution, where they accumulate exponentially fast (under a constant step size choice). 
More importantly, the limiting (expected) error bounds on the distance of the iterates from the optimal solution decrease with the network size, which is a comparable performance to a centralized stochastic gradient algorithm. 
Numerical examples further demonstrate the effectiveness of the method.

{\color{red}This is a preliminary version of the paper \cite{pu2018distributed2}.}
\end{abstract}

%%%%%%%%%%%%%%%%%%%%%%%%%%%%%%%%%%%%%%%%%%%%%%%%%%%%%%%%%%%%%%%%%%%%%%%%%%%%%%%%
	\section{Introduction}
Consider a set of agents $\mathcal{N}=\{1,2,\ldots,n\}$ connected over a network. Each agent has a local smooth and strongly convex cost function $f_i:\mathbb{R}^{p}\rightarrow \mathbb{R}$. The global objective is to locate $x\in\mathbb{R}^p$ that minimizes the average of all cost functions:
\begin{equation}
\min_{x\in \mathbb{R}^{p}}f(x)\left(=\frac{1}{n}\sum_{i=1}^{n}f_i(x)\right).  \label{opt Problem_def}
\end{equation}%
Scenarios in which problem (\ref{opt Problem_def}) is considered include distributed machine learning \cite{forrester2007multi,nedic2017fast,cohen2017projected,wai2018multi}, multi-agent target seeking \cite{pu2016noise,chen2012diffusion}, and wireless networks \cite{cohen2017distributed,mateos2012distributed,baingana2014proximal}, among many others.

% and $X\subseteq \mathbb{R}^{m}$. 
To solve problem (\ref{opt Problem_def}), we assume each agent $i$ queries a stochastic oracle ($\mathcal{SO}$) to obtain noisy gradient samples of the form $g_i(x,\xi_i)$ that satisfies the following condition:
\begin{assumption}
	\label{asp: gradient samples}
	For all $i\in\mathcal{N}$ and all $x\in\mathbb{R}^p$, 
	each random vector $\xi_i\in\mathbb{R}^m$ is independent, and
	\begin{equation}
	\begin{split}
	& \mathbb{E}_{\xi_i}[g_i(x,\xi_i)\mid x] =  \nabla f_i(x),\\
	& \mathbb{E}_{\xi_i}[\|g_i(x,\xi_i)-\nabla f_i(x)\|^2\mid x]  \le  \sigma^2\quad\hbox{\ for some $\sigma>0$}.
	\end{split}
	\end{equation}
\end{assumption}
The above assumption of stochastic gradients holds true for many on-line distributed learning problems, where $f_i(x)=\bE_{\xi_i}[F_i(x,\xi_i)]$ denotes the expected loss function agent $i$ wishes to minimize, while independent samples $\xi_i$ are gathered continuously over time.
For another example, in simulation-based optimization, the gradient estimation often incurs noise 
that can be due to various sources, such as modeling and discretization errors, 
incomplete convergence, and finite sample size for Monte-Carlo methods~\cite{kleijnen2008design}. 

Distributed algorithms dealing with problem (\ref{opt Problem_def}) have been studied extensively in the literature \cite{tsitsiklis1986distributed,nedic2009distributed,nedic2010constrained,jakovetic2014fast,kia2015distributed,shi2015extra,di2016next,qu2017harnessing,nedic2017achieving,pu2018push}.
%Stochastic optimization methods addressing Assumption \ref{asp: gradient samples} date back to the seminal works \cite{robbins1951stochastic} and \cite{kiefer1952stochastic}. 
Recently, there has been considerable interest in distributed implementation of stochastic gradient algorithms (see \cite{ram2010distributed,cavalcante2013distributed,towfic2014adaptive,lobel2011distributed,srivastava2011distributed,wang2015cooperative,lan2017communication,pu2017flocking,lian2017can,pu2018swarming}). The literature has shown that distributed algorithms may compete with, or even outperform, their centralized counterparts under certain conditions \cite{pu2017flocking,lian2017can,pu2018swarming}. For instance, in our recent work~\cite{pu2018swarming}, we proposed a swarming-based approach for distributed stochastic optimization which beats a centralized gradient method in real-time assuming that all $f_i$ are identical. 
However, to the best of our knowledge, there is no distributed stochastic gradient method addressing problem (\ref{opt Problem_def}) that shows comparable performance with a centralized approach. In particular, under constant step size policies none of the existing algorithms achieve an error bound that is decreasing in the network size $n$. 

%Recent works \cite{nedić2017achieving,qu2017harnessing} utilize a gradient tracking technique that allows the solutions converge linearly to the optimum under constant step sizes.

A distributed gradient tracking method was proposed in~\cite{di2016next,nedic2017achieving,qu2017harnessing}, where 
the agent-based auxiliary variables $y_i$ were introduced to track the average gradients of $f_i$ assuming accurate gradient information is available. It was shown that the method, with constant step size, generates iterates that converge linearly to the optimal solution.
Inspired by the approach, in this paper we consider a distributed stochastic gradient tracking method. By comparison, in our proposed algorithm $y_i$ are tracking the stochastic gradient averages of $f_i$.
We are able to show that the iterates generated by each agent reach, in expectation, a neighborhood of the optimal point exponentially fast under a constant step size.
Interestingly, with a sufficiently small step size, the limiting error bounds on the distance 
between agent iterates and the optimal solution decrease in the network size $n$, which is comparable to the performance of a centralized stochastic gradient algorithm.

Our work is also related to the extensive literature in stochastic approximation (SA) methods dating back to the seminal works~\cite{robbins1951stochastic} and~\cite{kiefer1952stochastic}. These works include the analysis of convergence (conditions for convergence, rates of convergence, suitable choice of step size) in the context of diverse noise models~\cite{kushner2003stochastic}.

The paper is organized as follows. In Section~\ref{sec: set}, we introduce the distributed stochastic gradient tracking method along with the main results. We perform analysis in Section~\ref{sec:cohesion} 
and provide a numerical example in Section~\ref{sec: simulation} to illustrate our theoretical findings. 
Section~\ref{sec: conclusion} concludes the paper.

%Recent progress include \cite{ghadimi-a,ghadimi-b,roux} and \cite{hennig}. 

\subsection{Notation}
\label{subsec:pre}
Throughout the paper, vectors default to columns if not otherwise specified.
Let each agent $i$ hold a local copy $x_i\in\mathbb{R}^p$ of the decision variable and an auxiliary variable $y_i\in\mathbb{R}^p$. Their values at iteration/time $k$ are denoted by $x_{i,k}$ and $y_{i,k}$, respectively. 
We let
\begin{equation*}
 \mx := [x_1, x_2, \ldots, x_n]^{\T}\in\mathbb{R}^{n\times p},\ \ \my := [y_1, y_2, \ldots, y_n]^{\T}\in\mathbb{R}^{n\times p},
\end{equation*}
%and
\begin{equation}
\ox :=  \frac{1}{n}\mathbf{1}^{\T} \mx\in\mathbb{R}^{1\times p},\ \ \oy :=  \frac{1}{n}\mathbf{1}^{\T}\my\in\mathbb{R}^{1\times p},
\end{equation}
where $\mathbf{1}$ denotes the vector with all entries equal to 1.
We define an aggregate objective function of the local variables:
\begin{equation}
F(\mx):=\sum_{i=1}^nf_i(x_i),
\end{equation}
and write
\begin{equation*}
\nabla F(\mx):=\left[\nabla f_1(x_1), \nabla f_2(x_2), \ldots, \nabla f_n(x_n)\right]^{\T}\in\mathbb{R}^{n\times p}.
\end{equation*}
In addition, let
\begin{equation}
\label{def: h}
h(\mx):=\frac{1}{n}\mathbf{1}^{\T}\nabla F(\mx)\in\mathbb{R}^{1\times p},
\end{equation}
\begin{equation*}
\boldsymbol{\xi} := [\xi_1, \xi_2, \ldots, \xi_n]^{\T}\in\mathbb{R}^{n\times p},
\end{equation*}
and
\begin{equation}
G(\mx,\boldsymbol{\xi}):=[g_1(x_1,\xi_1), g_2(x_2,\xi_2), \ldots, g_n(x_n,\xi_n)]^{\T}\in\mathbb{R}^{n\times p}.
\end{equation}

Inner product of two vectors $a,b$ of the same dimension is written as $\langle a,b\rangle$. For two matrices $A,B\in\mathbb{R}^{n\times p}$, we define
\begin{equation}
\langle A,B\rangle :=\sum_{i=1}^n\langle A_i,B_i\rangle,
\end{equation}
where $A_i$ (respectively, $B_i$) represents the $i$-th row of $A$ (respectively, $B$). We use $\|\cdot\|$ to denote the $2$-norm of vectors; for matrices, $\|\cdot\|$ denotes the Frobenius norm.

A graph is a pair $\mathcal{G}=(\mathcal{V},\mathcal{E})$ where $\mathcal{V}=\{1,2,\ldots,n\}$ is the set of vertices (nodes) and $\mathcal{E}\subseteq \mathcal{V}\times \mathcal{V}$ represents the set of edges connecting vertices. We assume agents communicate in an undirected graph, i.e., $(i,j)\in\mathcal{E}$ iff $(j,i)\in\mathcal{E}$.
Denote by $W=[w_{ij}]\in\mathbb{R}^{n\times n}$ the coupling matrix of agents. Agent $i$ and $j$  are connected iff $w_{ij}=w_{ji}>0$ ($w_{ij}=w_{ji}=0$ otherwise). Formally, we assume the following condition regarding the interaction among agents:
\begin{assumption}
	\label{asp: network}
	The graph $\mathcal{G}$ corresponding to the network of agents is undirected and connected (there exists a path between any two agents). Nonnegative coupling matrix $W$ is doubly stochastic, 
	i.e., $W\mathbf{1}=\mathbf{1}$ and $\mathbf{1}^{\T}W=\mathbf{1}^{\T}$.
	In addition, $w_{ii}>0$ for some $i\in\mathcal{N}$.
\end{assumption}
%In view of this assumption, due to being symmetric, the matrix $W$ is doubly stochastic, i.e., $\mathbf{1}^{\T}W=W\mathbf{1}=\mathbf{1}$.
We will frequently use the following result, which is a direct implication of Assumption \ref{asp: network} (see \cite{qu2017harnessing} Section II-B):
\begin{lemma}
	\label{lem: spectral norm}
	Let Assumption \ref{asp: network} hold, and let $\rho_w$ denote the spectral norm of 
	the matrix $W-\frac{1}{n}\mathbf{1}\mathbf{1}^{\T}$. Then, $\rho_w<1$ and 
	\begin{equation*}
	\|W\omega-\mathbf{1}\overline{\omega}\|\le \rho_w\|\omega-\mathbf{1}\overline{\omega}\|
	\end{equation*}
	for all $\omega\in\mathbb{R}^{n\times p}$, where $\overline{\omega}=\frac{1}{n}\mathbf{1}^{\T}\omega$.
\end{lemma}

\section{A Distributed Stochastic Gradient Tracking Method}
\label{sec: set}
We consider the following distributed stochastic gradient tracking method: at each step $k\in\mathbb{N}$, 
every agent $i$ independently implements the following two steps:
\begin{equation}
\label{eq: x_i,k}
\begin{split}
x_{i,k+1} = & \sum_{j=1}^{n}w_{ij}(x_{j,k}-\alpha y_{j,k}), \\
y_{i,k+1} = & \sum_{j=1}^{n}w_{ij}y_{j,k}+g_i(x_{i,k+1},\xi_{i,k+1})-g_i(x_{i,k},\xi_{i,k}),
\end{split}
\end{equation}
where $\alpha>0$ is a constant step size. The iterates are initiated with an arbitrary 
$x_{i,0}$ and $y_{i,0}= g_i(x_{i,0},\xi_{i,0})$ for all~$i\in{\cal N}$.
We can also write (\ref{eq: x_i,k}) in the following compact form:
\begin{equation}
\label{eq: x_k}
\begin{split}
\mx_{k+1} = & W(\mx_k-\alpha \my_k), \\
\my_{k+1} = & W\my_k+G(\mx_{k+1},\boldsymbol{\xi}_{k+1})-G(\mx_k,\boldsymbol{\xi}_k).
\end{split}
\end{equation}
Algorithm (\ref{eq: x_i,k}) is closely related to the schemes considered in \cite{di2016next,nedic2017achieving,qu2017harnessing}, where auxiliary variables $y_{i,k}$ 
were introduced to track the average $\frac{1}{n}\sum_{i=1}^{n}\nabla f_i(x_{i,k})$. This design ensures that the algorithm achieves linear convergence under constant step size choice.
Correspondingly, under our approach $y_{i,k}$ are (approximately) tracking $\frac{1}{n}\sum_{i=1}^{n}g_i(x_{i,k},\xi_{i,k})$.
% since we are using stochastic gradients $g_i(x_{i,k},\xi_{i,k})$ instead of exact ones $\nabla f_i(x_{i,k})$. 
To see why this is the case, note that
\begin{equation}
\oy_k = \frac{1}{n}\mathbf{1}^{\T} \my_k.
\end{equation}
Since $y_{i,0}=g(x_{i,0},\xi_{i,0}),\forall i$. By induction,
\begin{equation}
\oy_k=\frac{1}{n}\mathbf{1}^{\T}G(\mx_k,\boldsymbol{\xi}_k),\forall k.
\end{equation}
We will show that $\my_k$ is close to $\mathbf{1}\oy_k$ at each round. Hence $y_{i,k}$ are (approximately) tracking $\frac{1}{n}\sum_{i=1}^{n}g_i(x_{i,k},\xi_{i,k})$.

Throughout the paper, we make the following standing assumption on the objective functions $f_i$:
\begin{assumption}
	\label{asp: strconvexity}
	Each $f_i:\mathbb{R}^p\rightarrow \mathbb{R}$ is  $\mu$-strongly convex with $L$-Lipschitz continuous gradients, i.e., for any $x,x'\in\mathbb{R}^p$,
	\begin{equation}
	\begin{split}
	& \langle \nabla f_i(x)-\nabla f_i(x'),x-x'\rangle\ge \mu\|x-x'\|^2,\\
	& \|\nabla f_i(x)-\nabla f_i(x')\|\le L \|x-x'\|.
	\end{split}
	\end{equation}
\end{assumption}
Under Assumption \ref{asp: strconvexity}, problem (\ref{opt Problem_def}) has a unique solution denoted by $x^*\in\mathbb{R}^{1\times p}$.
%Here $W=[w_{ij}]$ is the coupling matrix assumed to be doubly stochastic. 

\subsection{Main Results}
\label{subsec: main}
Main convergence properties of the distributed gradient tracking method (\ref{eq: x_i,k}) are covered in the following theorem.
\begin{theorem}
	\label{Theorem}
	Suppose Assumptions \ref{asp: gradient samples}-\ref{asp: strconvexity} hold and $\alpha$ satisfies
	\begin{equation}
		\label{alpha_ultimate_bound}
		\alpha\le \min\left\{\frac{(1-\rho_w^2)}{12\rho_w^2 L},\frac{(1-\rho_w^2)^2}{2\sqrt{\Gamma}L\max\{6\rho_w\|W-I\|,1-\rho_w^2\}}, \frac{(1-\rho_w^2)}{3\rho_w^{2/3}L}\left[\frac{\mu^2}{L^2}\frac{(\Gamma-1)}{\Gamma(\Gamma+1)}\right]^{1/3}\right\}
		\end{equation}
	for some $\Gamma>1$. Then both $\sup_{l\ge k}\bE[\|\ox_l-x^*\|^2]$ and $\sup_{l\ge k}\bE[\|\mx_{l+1}-\mathbf{1}\ox_{l+1}\|^2]$ converge at the linear rate $\mathcal{O}(\rho_A^k)$, where $\rho_A<1$ is the spectral radius of
	\begin{eqnarray*}
			A=\begin{bmatrix}
				1-\alpha\mu & \frac{\alpha L^2}{\mu n}(1+\alpha\mu) & 0\\
				0 & \frac{1}{2}(1+\rho_w^2) & \alpha^2\frac{(1+\rho_w^2)\rho_w^2}{(1-\rho_w^2)}\\
				2\alpha nL^3 & \left(\frac{1}{\beta}+2\right)\|W-I\|^2 L^2+3\alpha L^3 & \frac{1}{2}(1+\rho_w^2)
			\end{bmatrix},
	\end{eqnarray*}
in which $\beta=\frac{1-\rho_w^2}{2\rho_w^2}-4\alpha L-2\alpha^2 L^2$.
Furthermore,
\begin{equation}
	\label{error_bound_ultimate}
	\limsup_{k\rightarrow\infty}\bE[\|\ox_k-x^*\|^2]\le \frac{(\Gamma+1)}{\Gamma}\frac{\alpha\sigma^2}{\mu n}
	+\left(\frac{\Gamma+1}{\Gamma-1}\right)\frac{4\alpha^2 L^2(1+\alpha\mu)(1+\rho_w^2)\rho_w^2}{\mu^2 n(1-\rho_w^2)^3}M_{\sigma},
	\end{equation}
and
\begin{equation}
	\label{consensus_error_bound_ultimate}
	\limsup_{k\rightarrow\infty}\bE[\|\mx_k-\mathbf{1}\ox_k\|^2]
	\le \left(\frac{\Gamma+1}{\Gamma-1}\right)\frac{4\alpha^2 (1+\rho_w^2)\rho_w^2(2\alpha^2L^3\sigma^2+\mu M_{\sigma})}{\mu(1-\rho_w^2)^3},
	\end{equation}
where 
\begin{equation}
	\label{M_sigma}
	M_{\sigma}:=\left[3\alpha^2L^2+2(\alpha L+1)(n+1)\right]\sigma^2.
	\end{equation}
\end{theorem}
\begin{remark}
	The first term on the right hand side of (\ref{error_bound_ultimate}) can be interpreted as the error caused by stochastic gradients only, since it does not depend on the network topology. The second term as well as the bound in (\ref{consensus_error_bound_ultimate}) are network dependent and increase with $\rho_w$ (larger $\rho_w$ indicates worse network connectivity).
	
	In light of (\ref{error_bound_ultimate}) and (\ref{consensus_error_bound_ultimate}),
	\begin{equation*}
	\limsup_{k\rightarrow\infty}\bE[\|\ox_k-x^*\|^2]=\alpha\mathcal{O}\left(\frac{\sigma^2}{\mu n}\right)+\alpha^2\mathcal{O}\left(\frac{ L^2\sigma^2}{\mu^2}\right),
	\end{equation*}
	and
	\begin{equation*}
	\limsup_{k\rightarrow\infty}\frac{1}{n}\bE[\|\mx_k-\mathbf{1}\ox_k\|^2]=\alpha^4\mathcal{O}\left(\frac{ L^3\sigma^2}{\mu n}\right)+\alpha^2\mathcal{O}\left(\sigma^2\right).
	\end{equation*}
	Let $(1/n)\bE[\|\mx_k-\mathbf{1}x^*\|^2]$ measure the average quality of solutions obtained by all the agents. We have
		\begin{equation*}
	\limsup_{k\rightarrow\infty}\frac{1}{n}\bE[\|\mx_k-\mathbf{1}x^*\|^2]=\alpha\mathcal{O}\left(\frac{\sigma^2}{\mu n}\right)+\alpha^2\mathcal{O}\left(\frac{ L^2\sigma^2}{\mu^2}\right),
	\end{equation*}
	which is decreasing in the network size $n$ when $\alpha$ is sufficiently small\footnote{Although $\rho_w$ is also related to the network size $n$, it only appears in the terms with high orders of $\alpha$.}.
	
	Under a centralized algorithm in the form of
	\begin{equation}
	\label{eq: centralized}
	x_{k+1}=x_k-\alpha \frac{1}{n}\sum_{i=1}^n g_i(x_k,\xi_{i,k}),\ \ k\in\mathbb{N},
	\end{equation}
	we would obtain
	\begin{equation*}
\limsup_{k\rightarrow\infty}\bE[\|x_k-x^*\|^2]=\alpha\mathcal{O}\left(\frac{\sigma^2}{\mu n}\right).
	\end{equation*}
	It can be seen that the distributed stochastic gradient tracking method (\ref{eq: x_i,k}) is comparable with the centralized algorithm (\ref{eq: centralized}) in their ultimate error bounds (up to constant factors) with sufficiently small step sizes.
\end{remark}

\begin{corollary}
	\label{cor: speed}
	Under the conditions in Theorem {\ref{Theorem}}. Suppose in addition\footnote{This condition is weaker than (\ref{alpha_ultimate_bound}) in most cases.} 
	\begin{equation}
		\label{alpha condition corollary}
	\alpha\le \frac{(\Gamma+1)}{\Gamma}\frac{(1-\rho_w^2)}{8\mu}.
	\end{equation}
    Then
	\begin{equation*}
	\rho_A\le 1-\left(\frac{\Gamma-1}{\Gamma+1}\right)\alpha\mu.
	\end{equation*}
\end{corollary}
\begin{remark}
	Corollary \ref{cor: speed} implies that, for sufficiently small step sizes, 
	the distributed gradient tracking method has a comparable convergence speed to that of a centralized scheme (in which case the linear rate is $\mathcal{O}(1-2\alpha\mu)^k$).
\end{remark}
%A possible choice is as follows:
%\begin{eqnarray}
%x_{i,k+1}
%= & x_{i,k}-\tg\sum_{j=1}^{n}a_{ij}(x_{i,k}-x_{j,k})-\alpha y_{i,k}, \\
%y_{i,k+1}= & \sum_{j=1}^{n}v_{ij}y_{j,k}+g(x_{i,k},\xi_{i,k+1})-g(x_{i,k},\xi_{i,k}),
%\end{eqnarray}
%where $a_{ij}\in\{0,1\},\forall i,j$, and $A=[a_{ij}]$ denotes the adjacency matrix. In this case, $W=I-\tg \mathcal{L}$ where $\mathcal{L}$ stands for the graph Laplacian.
%Step size $\alpha$ and coefficient $\tg$ are to be set later.

\section{Analysis}
\label{sec:cohesion}

In this section, we prove Theorem \ref{Theorem} by studying the evolution of $\bE[\|\ox_k-x^*\|^2]$, $\bE[\|\mx_k-\mathbf{1}\ox_k\|^2]$ and $\bE[\|\my_k-\mathbf{1}\oy_k\|^2]$. Our strategy is to bound the three expressions in terms of linear combinations of their past values, in which way we establish a linear system of inequalities. This approach is different from those employed in \cite{qu2017harnessing,nedic2017achieving}, where the analyses pertain to the examination of $\|\ox_k-x^*\|$, $\|\mx_k-\mathbf{1}\ox_k\|$ and $\|\my_k-\mathbf{1}\oy_k\|$. Such distinction is due to the  stochastic gradients $g_i(x_{i,k},\xi_{i,k})$ whose variances play a crucial role in deriving the main inequalities.

We first introduce some lemmas that will be used later in the analysis. Denote by $\mathcal{H}_k$ the history sequence $\{\mx_0,\boldsymbol{\xi}_0,\my_0,\ldots,\mx_{k-1},\boldsymbol{\xi}_{k-1},\my_{k-1},\mx_k\}$, and define $\bE[\cdot \mid\mathcal{H}_k]$ as the conditional expectation given $\mathcal{H}_k$.
\begin{lemma}
	\label{lem: oy_k-h_k}
	Under Assumption \ref{asp: gradient samples},
	\begin{align}
	\bE\left[\|\oy_k-h(\mx_k)\|^2\mid\mathcal{H}_k\right] \le \frac{\sigma^2}{n}.
	\end{align}
\end{lemma}
\begin{proof}
	By the definitions of $\oy_k$ and $h(\mx_k)$,
	\begin{equation*}
		\bE\left[\|\oy_k-h(\mx_k)\|^2\mid\mathcal{H}_k\right]\\
		=	\frac{1}{n^2}\sum_{i=1}^n\bE\left[\|g_i(x_{i,k},\xi_{i,k})-\nabla f_i(x_{i,k})\|^2\vert\mathcal{H}_k\right]\le \frac{\sigma^2}{n}.
		\end{equation*}
\end{proof}
\begin{lemma}
	\label{lem: strong_convexity}
	Under Assumption \ref{asp: strconvexity},
	\begin{align}
	\| \nabla F(\ox_k)-h(\mx_k)\| \le \frac{L}{\sqrt{n}}\|\mx_k-\mathbf{1}\ox_k\|.
	\end{align}
	Suppose in addition $\alpha<2/(\mu+L)$. Then
	\begin{equation*}
	\|x-\alpha\nabla f(x)-x^*\|\le (1-\alpha \mu)\|x-x^*\|,\, \forall x\in\mathbb{R}^p.
	\end{equation*}
\end{lemma}
\begin{proof}
	See  \cite{qu2017harnessing} Lemma 10 for reference.
\end{proof}

In the following lemma, we establish bounds on $\|\mx_{k+1}-\mathbf{1}\ox_{k+1}\|^2$ and on the conditional expectations of $\|\ox_{k+1}-x^*\|^2$ and $\|\my_{k+1}-\mathbf{1}\oy_{k+1}\|^2$, respectively.
\begin{lemma}
	\label{lem: Main_Inequalities}
	Suppose Assumptions \ref{asp: gradient samples}-\ref{asp: strconvexity} hold and $\alpha<2/(\mu+L)$. We have the following inequalities:
\begin{equation}
		\label{First_Main_Inequality}
	\bE[\|\ox_{k+1}-x^*\|^2\mid \mathcal{H}_k]
	\le \left(1-\alpha\mu\right)\|\ox_k-x^*\|^2\\
	+\frac{\alpha L^2}{\mu n}\left(1+\alpha\mu\right)\|  \mx_k-\mathbf{1}\ox_k\|^2
	+\frac{\alpha^2\sigma^2}{n},
\end{equation}
	\begin{equation}
\label{Second_Main_Inequality}
\|\mx_{k+1}-\mathbf{1}\ox_{k+1}\|^2
\le \frac{(1+\rho_w^2)}{2}\|\mx_k-\mathbf{1}\ox_k\|^2\\+\alpha^2\frac{(1+\rho_w^2)\rho_w^2}{(1-\rho_w^2)}\|\my_{k}-\mathbf{1}\oy_k\|^2,
\end{equation}
and for any $\beta>0$,
\begin{multline}
\label{Third_Main_Inquality}
\bE[\|\my_{k+1}-\mathbf{1}\oy_{k+1}\|^2\mid \mathcal{H}_k]
\le \left(1+4\alpha L+2\alpha^2 L^2+\beta\right)\rho_w^2\bE[\|\my_{k}-\mathbf{1}\oy_{k}\|^2\mid \mathcal{H}_k]\\
+\left(\frac{1}{\beta}\|W-I\|^2 L^2+2\|W-I\|^2 L^2+3\alpha L^3\right)\|\mx_k-\mathbf{1}\ox_k\|^2+2\alpha nL^3\|\ox_k-x^*\|^2+M_{\sigma}.
\end{multline}
\end{lemma}
\begin{proof}
	See Appendix \ref{proof lem: Main_Inequalities}.
	\end{proof}

\subsection{Proof of Theorem \ref{Theorem}}
	Taking full expectation on both sides of (\ref{First_Main_Inequality}), (\ref{Second_Main_Inequality}) and (\ref{Third_Main_Inquality}), we obtain the following linear system of inequalities
\begin{eqnarray}
\label{linear_system}
\begin{bmatrix}
\bE[\|\ox_{k+1}-x^*\|^2]\\
\bE[\|\mx_{k+1}-\mathbf{1}\ox_{k+1}\|^2]\\
\bE[\|\my_{k+1}-\mathbf{1}\oy_{k+1}\|^2]
\end{bmatrix}
\le 
A\begin{bmatrix}
\bE[\|\ox_{k}-x^*\|^2]\\
\bE[\|\mx_{k}-\mathbf{1}\ox_{k}\|^2]\\
\bE[\|\my_{k}-\mathbf{1}\oy_{k}\|^2]
\end{bmatrix}+\begin{bmatrix}
\frac{\alpha^2\sigma^2}{n}\\
0\\
M_{\sigma}
\end{bmatrix},
\end{eqnarray}
where the inequality is to be taken component-wise, and the entries of the matrix
$A=[a_{ij}]$ are given by
	\begin{eqnarray*}
		& \begin{bmatrix}
			a_{11}\\
			a_{21}\\
			a_{31}
		\end{bmatrix}  =  
		\begin{bmatrix}
			1-\alpha\mu\\
			0\\
			2\alpha nL^3 
		\end{bmatrix},
		\begin{bmatrix}
			a_{12}\\
			a_{22}\\
			a_{32}
		\end{bmatrix} = 
		\begin{bmatrix}
			\frac{\alpha L^2}{\mu n}(1+\alpha\mu)\\
			\frac{1}{2}(1+\rho_w^2)\\
			\left(\frac{1}{\beta}+2\right)\|W-I\|^2 L^2+3\alpha L^3
		\end{bmatrix},\\
		& \begin{bmatrix}
			a_{13}\\
			a_{23}\\
			a_{33}
		\end{bmatrix} = 
		\begin{bmatrix}
			0 \\
			\alpha^2\frac{(1+\rho_w^2)\rho_w^2}{(1-\rho_w^2)}\\
			\left(1+4\alpha L+2\alpha^2 L^2+\beta\right)\rho_w^2
		\end{bmatrix},
\end{eqnarray*}
and $M_{\sigma}$ is given in (\ref{M_sigma}).
Hence $\sup_{l\ge k}\bE[\|\ox_l-x^*\|^2]$, $\sup_{l\ge k}\bE[\|\mx_l-\mathbf{1}\ox_l\|^2]$ and $\sup_{l\ge k}\bE[\|\my_l-\mathbf{1}\oy_l\|^2]$ all converge to a neighborhood of $0$ at the linear rate $\mathcal{O}(\rho_A^k)$ if the spectral radius of $A$ satisfies $\rho_A<1$. 
The next lemma provides conditions for relation $\rho_A<1$ to hold.
\begin{lemma}
	\label{lem: rho_M}
	Let $M=[m_{ij}]\in\mathbb{R}^{3\times 3}$ be a nonnegative, irreducible matrix with 
	$m_{ii}<\lambda^*$ for some~{$\lambda^*>0$ and all $i=1,2,3.$} 
	A necessary and sufficient condition for $\rho_M<\lambda^*$ is $\text{det}(\lambda^* I-M)>0$.
\end{lemma}
\begin{proof}
	See Appendix \ref{subsec: proof lemma rho_M}.
	\end{proof}
	Let $\alpha$ and $\beta$ be such that the following relations hold\footnote{Matrix $A$ in Theorem~\ref{Theorem} 
	corresponds to such a choice of $\alpha$ and $\beta$.}.
\begin{equation}
	\label{beta}
	a_{33}=\left(1+4\alpha L+2\alpha^2 L^2+\beta\right)\rho_w^2=\frac{1+\rho_w^2}{2}<1,
	\end{equation}
	\begin{equation}
	\label{alpha,beta}
	a_{23}a_{32}=\alpha^2\frac{(1+\rho_w^2)\rho_w^2}{(1-\rho_w^2)}\left[\left(\frac{1}{\beta}+2\right)\|W-I\|^2 L^2+3\alpha L^3\right]
	\le\frac{1}{\Gamma}(1-a_{22})(1-a_{33})
	\end{equation}
for some $\Gamma>1$, and
\begin{equation}
	\label{alpha last condition}
	a_{12}a_{23}a_{31}=\frac{2\alpha^4 L^5(1+\alpha\mu)}{\mu}\frac{(1+\rho_w^2)}{(1-\rho_w^2)}\rho_w^2
	\le \frac{1}{\Gamma+1}(1-a_{11})[(1-a_{22})(1-a_{33})-a_{23}a_{32}].
	\end{equation}
Then,
\begin{multline*}
	\text{det}(I-A)=(1-a_{11})(1-a_{22})(1-a_{33})-(1-a_{11})a_{23}a_{32}-a_{12}a_{23}a_{31}\\
	\ge \frac{\Gamma}{(\Gamma+1)}(1-a_{11})[(1-a_{22})(1-a_{33})-a_{23}a_{32}]
	\ge \left(\frac{\Gamma-1}{\Gamma+1}\right)(1-a_{11})(1-a_{22})(1-a_{33})>0,
	\end{multline*}
given that $a_{11},a_{22},a_{33}<1$. In light of Lemma \ref{lem: rho_M}, $\rho_A<1$.
In addition, denoting $B:=[\frac{\alpha^2\sigma^2}{n}, 0, M_{\sigma}]^{\T}$, we get 
%	Suppose at the moment $a_{11},a_{22},a_{33}<1$ and $\rho_A<1$. In light of Lemma \ref{lem: rho_M}, $\text{det}(I-A)>0$. Then
	\begin{align}
		\label{error_bound_preliminary}
	\lim\sup_{k\rightarrow\infty}\bE[\|\ox_k-x^*\|^2]\le & [(I-A)^{-1}B]_1 \notag\\
	= & \frac{1}{\text{det}(I-A)}\left\{\left[(1-a_{22})(1-a_{33})-a_{23}a_{32}\right]\frac{\alpha^2\sigma^2}{n}+a_{12}a_{23}M_{\sigma}\right\} \notag\\
	\le & \frac{(\Gamma+1)}{\Gamma}\frac{\alpha\sigma^2}{\mu n}+\left(\frac{\Gamma+1}{\Gamma-1}\right)\frac{ a_{12}a_{23}M_{\sigma}}{(1-a_{11})(1-a_{22})(1-a_{33})}\notag\\
	= & \frac{(\Gamma+1)}{\Gamma}\frac{\alpha\sigma^2}{\mu n}+ \left(\frac{\Gamma+1}{\Gamma-1}\right)\frac{\alpha^3 L^2(1+\alpha\mu)(1+\rho_w^2)\rho_w^2M_{\sigma}}{\mu n(1-\rho_w^2)(1-a_{11})(1-a_{22})(1-a_{33})}\notag\\
	= & \frac{(\Gamma+1)}{\Gamma}\frac{\alpha\sigma^2}{\mu n}+\left(\frac{\Gamma+1}{\Gamma-1}\right)\frac{4\alpha^2 L^2(1+\alpha\mu)(1+\rho_w^2)\rho_w^2}{\mu^2 n(1-\rho_w^2)^3}M_{\sigma},
	\end{align}
and
	\begin{multline*}
	\lim\sup_{k\rightarrow\infty}\bE[\|\mx_k-\mathbf{1}\ox_k\|^2] \le [(I-A)^{-1}B]_2
	=\frac{1}{\text{det}(I-A)}\left[a_{23}a_{31}\frac{\alpha^2\sigma^2}{n}+a_{23}(1-a_{11})M_{\sigma}\right]\\
	\le \left(\frac{\Gamma+1}{\Gamma-1}\right)\frac{a_{23}}{(1-a_{11})(1-a_{22})(1-a_{33})}\left(2\alpha nL^3\frac{\alpha^2 \sigma^2}{n}+\alpha\mu M_{\sigma}\right)\\
	= \frac{4(\Gamma+1)\alpha^2 (1+\rho_w^2)\rho_w^2(2\alpha^2L^3\sigma^2+\mu M_{\sigma})}{(\Gamma-1)\mu(1-\rho_w^2)^3}.
	\end{multline*}

	We now show that (\ref{beta}), (\ref{alpha,beta}) and (\ref{alpha last condition}) are satisfied under condition (\ref{alpha_ultimate_bound}).  By (\ref{beta}) and (\ref{alpha,beta}), it follows that
	\begin{equation*}
		\beta=\frac{1-\rho_w^2}{2\rho_w^2}-4\alpha L-2\alpha^2 L^2>0,
		\end{equation*}
	and
		\begin{equation*}
		\alpha^2\frac{(1+\rho_w^2)\rho_w^2}{(1-\rho_w^2)}\left[\left(\frac{1}{\beta}+2\right)\|W-I\|^2 L^2+3\alpha L^3\right]\le\frac{(1-\rho_w^2)^2}{4\Gamma}.
		\end{equation*}
		Since by (\ref{alpha_ultimate_bound}) we have
		\begin{equation*}
			\alpha\le \frac{1-\rho_w^2}{12\rho_w^2 L},\qquad
			%\end{equation*}}\normalsize
		%{\small\begin{equation*}
		\beta\ge \frac{1-\rho_w^2}{8\rho_w^2}>0,
		\end{equation*}
		we only need to show that
		\begin{equation*}
		\alpha^2\left[\frac{(2+6\rho_w^2)}{(1-\rho_w^2)}\|W-I\|^2 L^2+\frac{(1-\rho_w^2)}{4\rho_w^2}L^2\right]\le\frac{(1-\rho_w^2)^3}{4\Gamma(1+\rho_w^2)\rho_w^2}.
		\end{equation*}
		The preceding inequality is equivalent to
		\begin{equation*}
		\alpha\le \frac{(1-\rho_w^2)^2}{L\sqrt{\Gamma(1+\rho_w^2)}\sqrt{4\rho_w^2(2+6\rho_w^2)\|W-I\|^2+(1-\rho_w^2)^2}},
		\end{equation*}
	implying that it is sufficient to have
	\begin{equation*}
	\alpha\le\frac{(1-\rho_w^2)^2}{2\sqrt{\Gamma}L\max(6\rho_w\|W-I\|,1-\rho_w^2)}.
	\end{equation*}
		To see that relation (\ref{alpha last condition}) holds, consider a stronger condition
		\begin{equation*}
		\frac{2\alpha^4 L^5(1+\alpha\mu)}{\mu}\frac{(1+\rho_w^2)}{(1-\rho_w^2)}\rho_w^2
		\le \frac{(\Gamma-1)}{\Gamma(\Gamma+1)}(1-a_{11})(1-a_{22})(1-a_{33}),
		\end{equation*}
		or equivalently,
		\begin{equation*}
		\frac{2\alpha^3 L^5(1+\alpha\mu)}{\mu^2}\frac{(1+\rho_w^2)}{(1-\rho_w^2)}\rho_w^2 \le \frac{(\Gamma-1)}{4\Gamma(\Gamma+1)}(1-\rho_w^2)^2.
		\end{equation*}
		It suffices that
		\begin{equation}
		\alpha\le \frac{(1-\rho_w^2)}{3\rho_w^{2/3}L}\left[\frac{\mu^2}{L^2}\frac{(\Gamma-1)}{\Gamma(\Gamma+1)}\right]^{1/3}.
		\end{equation}

\subsection{Proof of Corollary \ref{cor: speed}}
We derive an upper bound of $\rho_A$ under condition (\ref{alpha_ultimate_bound}) and (\ref{alpha condition corollary}). Note that the characteristic function of $A$ is given by
\begin{equation*}
\text{det}(\lambda I-A)=(\lambda-a_{11})(\lambda-a_{22})(\lambda-a_{33})
-(\lambda-a_{11})a_{23}a_{32}-a_{12}a_{23}a_{31}.
\end{equation*}
Since $\text{det}(I-A)> 0$ and $\text{det}(\max\{a_{11},a_{22},a_{33}\} I-A)<0$, $\rho_A\in(\max\{a_{11},a_{22},a_{33}\},1)$. By (\ref{alpha,beta}) and (\ref{alpha last condition}),
\begin{multline*}
\text{det}(\lambda I-A)
\ge (\lambda-a_{11})(\lambda-a_{22})(\lambda-a_{33})-(\lambda-a_{11})a_{23}a_{32}-\frac{1}{\Gamma+1}(1-a_{11})[(1-a_{22})(1-a_{33})-a_{23}a_{32}]\\
\ge (\lambda-a_{11})(\lambda-a_{22})(\lambda-a_{33})-\frac{1}{\Gamma}(\lambda-a_{11})(1-a_{22})(1-a_{33})
-\frac{(\Gamma-1)}{\Gamma(\Gamma+1)}(1-a_{11})(1-a_{22})(1-a_{33}).
\end{multline*}
Suppose $\lambda=1-\epsilon$ for some $\epsilon\in(0,\alpha\mu)$, satisfying
\begin{equation*}
\text{det}(\lambda I-A)
\ge \frac{1}{4}(\alpha\mu-\epsilon)\left(1-\rho_w^2-2\epsilon\right)^2
-\frac{1}{4\Gamma}(\alpha\mu-\epsilon)(1-\rho_w^2)^2-\frac{(\Gamma-1)\alpha\mu}{4\Gamma(\Gamma+1)}(1-\rho_w^2)^2\ge 0.
\end{equation*}
Under (\ref{alpha condition corollary}), it suffices that
\begin{equation*}
\epsilon\le \left(\frac{\Gamma-1}{\Gamma+1}\right)\alpha\mu.
\end{equation*}
Denote
\begin{equation*}
\tilde{\lambda}=1-\left(\frac{\Gamma-1}{\Gamma+1}\right)\alpha\mu.
\end{equation*}
Then $\text{det}(\tilde{\lambda} I-A)\ge 0$ so that $\rho_A\le \tilde{\lambda}$.

\section{Numerical Example}
\label{sec: simulation}

In this section, we provide a numerical example to illustrate our theoretic findings. 
Consider the \emph{on-line} Ridge regression problem, i.e.,
\begin{equation}
\label{Ridge Regression}
f(x):=\frac{1}{n}\sum_{i=1}^n\mathbb{E}_{u_i,v_i}\left[\left(u_i^{\T} x-v_i\right)^2+\rho\|x\|^2\right].
\end{equation}
where $\rho>0$ ia a penalty parameter.
For each agent $i$, samples in the form of $(u_i,v_i)$ are gathered continuously with $u_i\in\mathbb{R}^p$ representing the features and $v_i\in\mathbb{R}$ being the observed outputs. We assume that each $u_i\in[-1,1]^p$ is uniformly distributed, and $v_i$ is drawn according to $v_i=u_i^{\T} \tilde{x}_i+\varepsilon_i$. Here $\tilde{x}_i\in[0.4,0.6]^p$ is a predefined, (uniformly) randomly generated parameter, and $\varepsilon_i$ are independent Gaussian noises with mean $0$ and variance $0.25$.
Given a pair $(u_i,v_i)$, agent $i$ can calculate an estimated gradient of $f_i(x)$:
\begin{equation}
g_i(x,u_i,v_i)=2(u_i^{\T}x -v_i)u_i+2\rho x,
\end{equation}
which is unbiased.
Notice that the Hessian matrix of $f(x)$ is $\mathbf{H}_f=(2/3+2\rho)I_d\succ 0$. Therefore $f(\cdot)$ is strongly convex, and problem (\ref{Ridge Regression}) has a unique solution $x^*$ given by
\begin{equation*}
x^*=\frac{1}{(1+3\rho)}\sum_{i=1}^n\tilde{x}_i/n.
\end{equation*}

In the experiments, we consider $3$ instances with $p=20$ and $n\in\{10,25,100\}$, respectively. Under each instance, we draw $x_{i,0}$  uniformly randomly from $[5,10]^p$. Penalty parameter $\rho=0.01$ and step size $\gamma=0.01$. We assume that $n$ agents constitute a random network, in which each two agents are linked with probability $0.4$. The Metropolis rule is applied to define the weights $w_{ij}$ \cite{sayed2014adaptive}:
\begin{equation*}
w_{ij}=\begin{cases}
1/\max\{d_i,d_j\} & \text{if }i\in \mathcal{N}_i\setminus \{i\},  \\
1- \sum_{j\in\mathcal{N}_i}w_{ij} & \text{if }i=j,\\
0 & \text{if }i\notin \mathcal{N}_i.
\end{cases}
\end{equation*}
Here $d_i$ denotes the degree (number of ``neighbors'') of node $i$, and $\mathcal{N}_i$ is the set of ``neighbors''.

\begin{figure}
	\centering
	\subfigure[Instance $(p,n)=(20,10)$.]{\includegraphics[width=3.5in]{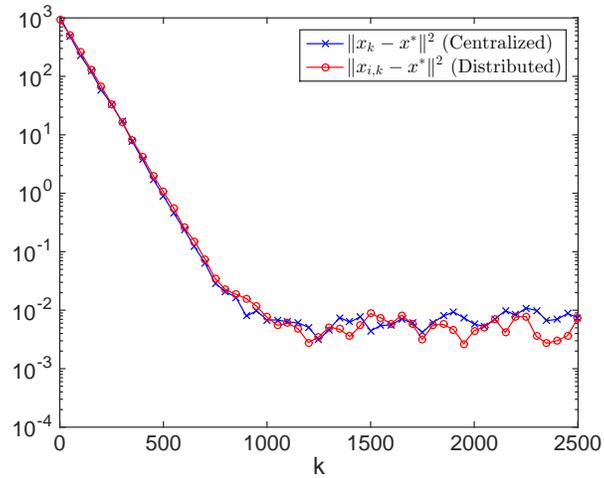}} 
	\subfigure[Instance $(p,n)=(20,25)$.]{\includegraphics[width=3.5in]{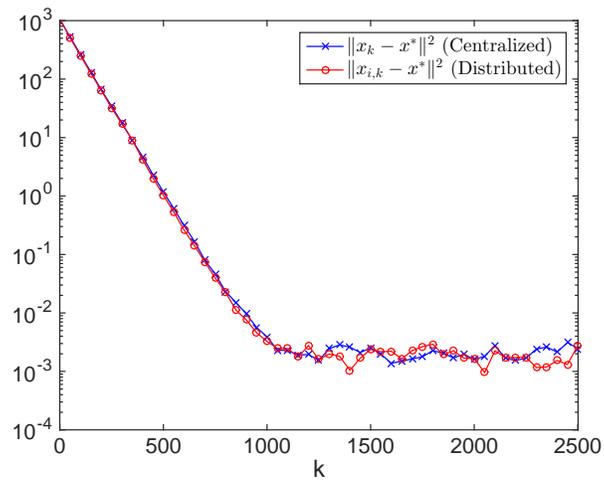}}
	\subfigure[Instance $(p,n)=(20,100)$.]{\includegraphics[width=3.5in]{{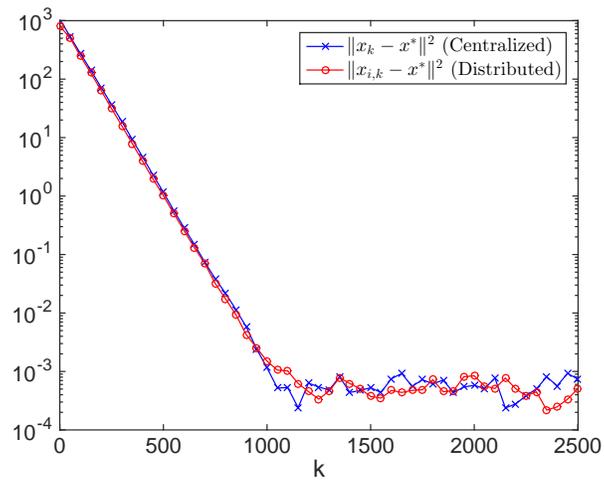}}} 
	\caption{Performance comparison between the distributed gradient tracking method and the centralized algorithm for on-line Ridge regression. For the decentralized method, the plots show the iterates generated 
	by a randomly selected node $i$ from the set $\mathcal{N}$.}
	\label{fig: comparison}
\end{figure}

In Figure \ref{fig: comparison}, we compare the performances of the distributed gradient tracking method (\ref{eq: x_i,k}) and the centralized algorithm (\ref{eq: centralized}) with the same parameters. It can be seen that the two approaches are comparable in their convergence speeds as well as the ultimate error bounds. Furthermore, the error bounds decrease in $n$ as expected from our theoretical analysis.
\section{Conclusions and Future Work}
\label{sec: conclusion}
This paper considers distributed multi-agent optimization over a network, where each agent only has access to inexact gradients of its local cost function. 
We propose a distributed stochastic gradient tracking method and show that the iterates obtained by each agent, using a constant step size value, reach a neighborhood of the optimum (in expectation) exponentially fast. More importantly, in a limit, the error bounds for the distances between the iterates and the optimal solution decrease in the network size, which is comparable with the performance of a centralized stochastic gradient algorithm. 
In our future work, we will consider adaptive step size policies, directed and/or time-varying interaction graphs, and more efficient communication protocols (e.g., gossip-based scheme).
%%%%%%%%%%%%%%%%%%%%%%%%%%%%%%%%%%%%%%%%%%%%%%%%%%%%%%%%%%%%%%%%%%%%%%%%%%%%%%%%

%%%%%%%%%%%%%%%%%%%%%%%%%%%%%%%%%%%%%%%%%%%%%%%%%%%%%%%%%%%%%%%%%%%%%%%%%%%%%%%%

%%%%%%%%%%%%%%%%%%%%%%%%%%%%%%%%%%%%%%%%%%%%%%%%%%%%%%%%%%%%%%%%%%%%%%%%%%%%%%%%

%%%%%%%%%%%%%%%%%%%%%%%%%%%%%%%%%%%%%%%%%%%%%%%%%%%%%%%%%%%%%%%%%%%%%%%%%%%%%%%%

\bibliographystyle{IEEEtran}
\bibliography{mybib}

\section{APPENDIX}

\subsection{Proof of Lemma \ref{lem: Main_Inequalities}}
\label{proof lem: Main_Inequalities}
	By (\ref{eq: x_i,k}),
\begin{equation}
\ox_{k+1}=\ox_k-\alpha \oy_k.
\end{equation}
It follows that
\begin{equation}
\|\ox_{k+1}-x^*\|^2=\|\ox_k-\alpha \oy_k-x^*\|^2
=\|\ox_k-x^*\|^2-2\alpha\langle \ox_k-x^*,\oy_k \rangle+\alpha^2\|\oy_k\|^2.
\end{equation}
Notice that $\bE[\oy_k\mid \mathcal{H}_k]=h(\mx_k)$, and
\begin{equation*}
\bE[\|\oy_k\|^2\mid\mathcal{H}_k]=\bE[\|\oy_k-h(\mx_k)\|^2\mid\mathcal{H}_k]+\|h(\mx_k)\|^2.
\end{equation*}
We have
\begin{multline}
\bE[\|\ox_{k+1}-x^*\|^2\mid \mathcal{H}_k]
=\|\ox_k-x^*\|^2-2\alpha\langle \ox_k-x^*,h(\mx_k) \rangle+\alpha^2\bE[\|\oy_k-h(\mx_k)\|^2\mid\mathcal{H}_k]\\
+\alpha^2\|h(\mx_k)\|^2
\le \|\ox_k-x^*\|^2-2\alpha\langle \ox_k-x^*,h(\mx_k) \rangle+\alpha^2\|h(\mx_k)\|^2+\frac{\alpha^2\sigma^2}{n},
\end{multline}
where the inequality follows from Lemma \ref{lem: oy_k-h_k}. Denote $\lambda=1-\alpha\mu$. In light of Lemma \ref{lem: strong_convexity},
\begin{align*}
& \bE[\|\ox_{k+1}-x^*\|^2\mid \mathcal{H}_k]\\
\le &\|\ox_k-x^*\|^2-2\alpha\langle \ox_k-x^*,\nabla F(\ox_k)\rangle+2\alpha\langle \ox_k-x^*,\nabla F(\ox_k)-h(\mx_k) \rangle\\
&+\alpha^2\|\nabla F(\ox_k)-h(\mx_k)\|^2+\alpha^2\|\nabla F(\ox_k)\|^2-2\alpha^2\langle \nabla F(\ox_k),\nabla F(\ox_k)-h(\mx_k)\rangle+\frac{\alpha^2\sigma^2}{n}\\
= &\|\ox_k-\alpha \nabla F(\ox_k)-x^*\|^2+\alpha^2\| \nabla F(\ox_k)-h(\mx_k)\|^2+\frac{\alpha^2\sigma^2}{n}\\
&+2\alpha \langle \ox_k-\alpha \nabla F(\ox_k)-x^*, \nabla F(\ox_k)-h(\mx_k)\rangle\\
\le &\lambda^2\|\ox_k-x^*\|^2+2\alpha \lambda\| \ox_k-x^*\| \| \nabla F(\ox_k)-h(\mx_k)\|+\alpha^2\| \nabla F(\ox_k)-h(\mx_k)\|^2+\frac{\alpha^2\sigma^2}{n}\\
\le &\lambda^2\|\ox_k-x^*\|^2+\frac{2\alpha \lambda L}{\sqrt{n}}\| \ox_k-x^*\|\| \mx_k-\mathbf{1}\ox_k\|+\frac{\alpha^2 L^2}{n}\|  \mx_k-\mathbf{1}\ox_k\|^2+\frac{\alpha^2\sigma^2}{n}\\
\le &\lambda^2\|\ox_k-x^*\|^2+\alpha\left(\lambda^2\mu\| \ox_k-x^*\|^2+\frac{L^2}{\mu n}\| \mx_k-\mathbf{1}\ox_k\|^2 \right)+\frac{\alpha^2 L^2}{n}\|  \mx_k-\mathbf{1}\ox_k\|^2+\frac{\alpha^2\sigma^2}{n}\\
= &\lambda^2\left(1+\alpha\mu\right)\|\ox_k-x^*\|^2
+\frac{\alpha L^2}{\mu n}\left(1+\alpha\mu\right)\|  \mx_k-\mathbf{1}\ox_k\|^2
+\frac{\alpha^2\sigma^2}{n}\\
\le &\left(1-\alpha\mu\right)\|\ox_k-x^*\|^2+\frac{\alpha L^2}{\mu n}\left(1+\alpha\mu\right)\|  \mx_k-\mathbf{1}\ox_k\|^2
+\frac{\alpha^2\sigma^2}{n}.
\end{align*}
Relation (\ref{Second_Main_Inequality}) follows from the following argument:
	\begin{multline}
\|\mx_{k+1}-\mathbf{1}\ox_{k+1}\|^2=\|W\mx_k-\alpha W\my_k-\mathbf{1}\ox_k+\alpha\mathbf{1}\oy_k\|^2\\
\le \|W\mx_k-\mathbf{1}\ox_k\|^2-2\alpha\langle W\mx_k-\mathbf{1}\ox_k, W\my_k-\mathbf{1}\oy_k\rangle+\alpha^2\|W\my_{k}-\mathbf{1}\oy_k\|^2\\
\le \rho_w^2\|\mx_k-\mathbf{1}\ox_k\|^2+\alpha\rho_w^2\left[\frac{(1-\rho_w^2)}{2\alpha\rho_w^2}\|\mx_k-\mathbf{1}\ox_k\|^2+\frac{2\alpha\rho_w^2}{(1-\rho_w^2)}\|\my_{k}-\mathbf{1}\oy_k\|^2\right]+\alpha^2\rho_w^2\|\my_{k}-\mathbf{1}\oy_k\|^2\\
\le \frac{(1+\rho_w^2)}{2}\|\mx_k-\mathbf{1}\ox_k\|^2+\alpha^2\frac{(1+\rho_w^2)\rho_w^2}{(1-\rho_w^2)}\|\my_{k}-\mathbf{1}\oy_k\|^2,
\end{multline}
where we used Lemma \ref{lem: spectral norm}.

To prove (\ref{Third_Main_Inquality}), we need some preparations first. For ease of exposition we will write $G_k:=G(\mx_k,\boldsymbol{\xi}_k)$ and $\nabla_k:=\nabla F(\mx_k)$ for short. From (\ref{eq: x_k}) and Lemma \ref{lem: spectral norm},
\begin{align*}
\|\my_{k+1}-\mathbf{1}\oy_{k+1}\|^2
= & \|W\my_k+G_{k+1}-G_k-\mathbf{1}\oy_k+\mathbf{1}(\oy_k-\oy_{k+1})\|^2\\
= & |W\my_k-\mathbf{1}\oy_k\|^2+\|G_{k+1}-G_k\|^2+n\|\oy_k-\oy_{k+1}\|^2+2\langle W\my_k-\mathbf{1}\oy_k,G_{k+1}-G_k\rangle\\
&+2\langle W\my_k-\mathbf{1}\oy_k,\mathbf{1}(\oy_k-\oy_{k+1})\rangle+2\langle G_{k+1}-G_k,\mathbf{1}(\oy_k-\oy_{k+1})\rangle\\
= &\|W\my_k-\mathbf{1}\oy_k\|^2+\|G_{k+1}-G_k\|^2-n\|\oy_k-\oy_{k+1}\|^2+2\langle W\my_k-\mathbf{1}\oy_k,G_{k+1}-G_k\rangle\\
\le &\rho_w^2\|\my_k-\mathbf{1}\oy_k\|^2+\|G_{k+1}-G_k\|^2+2\langle W\my_k-\mathbf{1}\oy_k,G_{k+1}-G_k\rangle.
\end{align*}
Notice that 
\begin{align*}
\bE[\|G_{k+1}-G_k\|^2\mid \mathcal{H}_k]= & \bE[\|\nabla_{k+1}-\nabla_k \|^2\mid \mathcal{H}_k]
+2\bE[\langle \nabla_{k+1}-\nabla_k , G_{k+1}-\nabla_{k+1}-G_k+\nabla_k \rangle \mid \mathcal{H}_k]\\
&+\bE[\|G_{k+1}-\nabla_{k+1}-G_k+\nabla_k \|^2\mid \mathcal{H}_k]\\
\le & \bE[\|\nabla_{k+1}-\nabla_k \|^2\mid \mathcal{H}_k]+2\bE[\langle \nabla_{k+1}, -G_k+\nabla_k \rangle \mid \mathcal{H}_k]+2n\sigma^2
\end{align*}
by Assumption \ref{asp: gradient samples},
and
\begin{multline*}
 \bE[\langle W\my_k-\mathbf{1}\oy_k,G_{k+1}-G_k\rangle\mid \mathcal{H}_k]
= \bE[\langle W\my_k-\mathbf{1}\oy_k,\nabla_{k+1}-G_k\rangle\mid \mathcal{H}_k]\\
= \bE[\langle W\my_k-\mathbf{1}\oy_k,\nabla_{k+1}-\nabla_k \rangle\mid \mathcal{H}_k]
+\bE[\langle W\my_k-\mathbf{1}\oy_k, -G_k+\nabla_k \rangle\mid \mathcal{H}_k].
%\le (1+\beta)\|W\my_k-\mathbf{1}\oy_k\|^2+(1+\frac{1}{\beta})\|G_{k+1}-G_k\|^2\\
%\le (1+\beta)\rho_w^2\|\my_k-\mathbf{1}\oy_k\|^2+2(1+\frac{1}{\beta})\|\nabla_{k+1}-\nabla_k \|^2
%+2(1+\frac{1}{\beta})\|G_{k+1}-\nabla_{k+1}-G_k+\nabla_k \|^2
\end{multline*}
We have
\begin{multline}
\label{my_k+1-oy_k+1 pre}
\bE[\|\my_{k+1}-\mathbf{1}\oy_{k+1}\|^2\mid\mathcal{H}_k]
\le \rho_w^2\bE[\|\my_k-\mathbf{1}\oy_k\|^2\mid \mathcal{H}_k]
+\bE[\|\nabla_{k+1}-\nabla_k \|^2\mid \mathcal{H}_k]\\
+2\bE[\langle \nabla_{k+1}, -G_k+\nabla_k \rangle \mid \mathcal{H}_k]
+2\bE[\langle W\my_k-\mathbf{1}\oy_k,\nabla_{k+1}-\nabla_k \rangle\mid \mathcal{H}_k]\\
+2\bE[\langle W\my_k-\mathbf{1}\oy_k, -G_k+\nabla_k \rangle\mid \mathcal{H}_k]+2n\sigma^2.
\end{multline}
Two additional lemmas are in hand.
\begin{lemma}
	\label{lem: 3 nablas}
	\begin{equation*}
	\bE[\langle \nabla_{k+1}, -G_k+\nabla_k \rangle \mid \mathcal{H}_k]\le \alpha L n\sigma^2.
	\end{equation*}
\end{lemma}
\begin{proof}
	From (\ref{eq: x_i,k}),
	\begin{equation}
	\label{nabla f_i expand}
	\nabla f_i(x_{i,k+1})=
	\nabla f_i\left(\sum_{i=1}^n w_{ij}x_{j,k}-\alpha\sum_{j=1}^n w_{ij}y_{j,k-1}-\alpha g_i(x_{i,k},\xi_{i,k})+\alpha g_i(x_{i,k-1},\xi_{i,k-1}))\right).
	\end{equation}
	Denote by $[x]_p$ the $p$-th component of any $x\in \mathbb{R}^p$. In light of Assumption \ref{asp: strconvexity},
	\begin{multline}
	\label{p-component Lipschitz}
	\Bigg\lvert[\nabla f_i(x_{i,k+1})]_q-\left[\nabla f_i\left(\sum_{i=1}^n w_{ij}x_{j,k}-\alpha\sum_{j=1}^n w_{ij}y_{j,k-1}-\alpha\nabla f_i(x_{i,k})+\alpha g_i(x_{i,k-1},\xi_{i,k-1})\right)\right]_p\Bigg\rvert\\
	\le \alpha L\lvert[g_i(x_{i,k},\xi_{i,k})-\nabla f_i(x_{i,k})]_q\rvert
	\end{multline}
	in each dimension $1\le q\le p$.
	Then,
	\begin{multline}
	\bE[\langle \nabla f_i(x_{i,k+1}), -g_i(x_{i,k},\xi_{i,k})+\nabla f_i(x_{i,k})\rangle \mid \mathcal{H}_k]\\
	=\sum_{q=1}^{p}\bE[\langle [\nabla f_i(x_{i,k+1})]_q, [-g_i(x_{i,k},\xi_{i,k})+\nabla f_i(x_{i,k})]_q\rangle \mid \mathcal{H}_k]\\
	\le \sum_{q=1}^{p}\alpha L \bE[|[g_i(x_{i,k},\xi_{i,k})-\nabla f_i(x_{i,k})]_q |^2\mid \mathcal{H}_k]
	=\alpha L \bE[\|g_i(x_{i,k},\xi_{i,k})-\nabla f_i(x_{i,k})\|^2\mid \mathcal{H}_k]\le \alpha L \sigma^2.
	\end{multline}
%	Since
%	{\footnotesize\begin{multline*}
%	\bE[\langle \nabla_{k+1}, -G_k+\nabla_k \rangle \mid \mathcal{H}_k]\\
%	=\sum_{i=1}^n\bE[\langle \nabla f_i(x_{i,k+1}), -g_i(x_{i,k},\xi_{i,k})+\nabla f_i(x_{i,k})\rangle \mid \mathcal{H}_k],
%	\end{multline*}}\normalsize
	The desired result then follows.
\end{proof}
%\begin{lemma}
%	\label{lem: 4 nablas square}
%	\begin{equation*}
%	\bE[\|G_{k+1}-\nabla_{k+1}-G_k+\nabla_k \|^2]\le 2n\sigma^2.
%	\end{equation*}
%\end{lemma}
%\begin{proof}
%	\begin{equation*}
%	\bE[\|G_{k+1}-\nabla_{k+1}-G_k+\nabla_k \|^2]=\bE[\|G_{k+1}-\nabla_{k+1}\|^2]+\bE[\|G_k-\nabla_k \|^2]\le 2n\sigma^2.
%	\end{equation*}
%\end{proof}
\begin{lemma}
	\label{lem: Vy_k-oy_k, 4 nablas}
	\begin{equation}
	\bE[\langle W\my_k-\mathbf{1}\oy_k, -G_k+\nabla_k \rangle\mid \mathcal{H}_k]\le \sigma^2.
	\end{equation}
\end{lemma}
\begin{proof}
	By (\ref{eq: x_i,k}),
	\begin{equation*}
	\bE[\langle W\my_k-\mathbf{1}\oy_k, -G_k+\nabla_k \rangle\mid \mathcal{H}_k]\\
	=\sum_{i=1}^n\bE\left[\Big\langle \sum_{j=1}^n w_{ij}y_{j,k}-\oy_k, \nabla f_i(x_{i,k})-g_i(x_{i,k},\xi_{i,k})\Big\rangle\Big\vert \mathcal{H}_k\right].
	\end{equation*}
	On one hand,
	\begin{multline*}
	\bE[\langle y_{j,k},\nabla f_i(x_{i,k})-g_i(x_{i,k},\xi_{i,k})\rangle\mid \mathcal{H}_k]\\
	=\bE\left[\Big\langle \sum_{n=1}^{n}v_{jn}y_{n,k-1}+\tilde{\nabla} f_j(x_{j,k})-\tilde{\nabla}f_j(x_{j,k-1}),\nabla f_i(x_{i,k})-g_i(x_{i,k},\xi_{i,k})\Big\rangle\Big\vert \mathcal{H}_k\right]\\
	=\bE\left[\left\langle\tilde{\nabla} f_j(x_{j,k}),\nabla f_i(x_{i,k})-g_i(x_{i,k},\xi_{i,k})\right\rangle\big\vert \mathcal{H}_k\right],
	\end{multline*}
	which gives
	\begin{multline*}
	\bE\left[\Big\langle \sum_{j=1}^n w_{ij}y_{j,k},\nabla f_i(x_{i,k})-g_i(x_{i,k},\xi_{i,k})\Big\rangle\Big\vert \mathcal{H}_k\right]\\
	=\bE[\langle w_{ii}g(x_{i,k},\xi_{i,k}),\nabla f_i(x_{i,k})-g_i(x_{i,k},\xi_{i,k})\rangle\mid \mathcal{H}_k]
	\le 0.
	\end{multline*}
	On the other hand,
\begin{multline*}
	\bE[\langle \oy_k, \nabla f_i(x_{i,k})-g_i(x_{i,k},\xi_{i,k})\rangle\mid \mathcal{H}_k]
	=\bE\left[\Big\langle \frac{1}{n}\sum_{j=1}^{n}\tilde{\nabla}f_j(x_{j,k}), \nabla f_i(x_{i,k})-g_i(x_{i,k},\xi_{i,k})\Big\rangle\Big\vert \mathcal{H}_k\right]\\
	=\bE\left[\Big\langle \frac{1}{n}g_i(x_{i,k},\xi_{i,k}), \nabla f_i(x_{i,k})-g_i(x_{i,k},\xi_{i,k})\Big\rangle\Big\vert \mathcal{H}_k\right].
	\end{multline*}
	We have
	\begin{equation}
	\bE[\langle W\my_k-\mathbf{1}\oy_k, -G_k+\nabla_k \rangle\mid \mathcal{H}_k]
	\le -\frac{1}{n}\sum_{i=1}^n\bE[\langle g(x_{i,k},\xi_{i,k}), \nabla f_i(x_{i,k})-g_i(x_{i,k},\xi_{i,k})\rangle\mid \mathcal{H}_k]
	\le \sigma^2.
	\end{equation}
\end{proof}
By (\ref{my_k+1-oy_k+1 pre}), Lemma \ref{lem: 3 nablas} and Lemma \ref{lem: Vy_k-oy_k, 4 nablas},
\begin{multline}
\label{my_k+1-oy_k+1}
\bE[\|\my_{k+1}-\mathbf{1}\oy_{k+1}\|^2\mid \mathcal{H}_k]\le \rho_w^2\bE[\|\my_{k}-\mathbf{1}\oy_{k}\|^2\mid \mathcal{H}_k]
+\bE[\|\nabla_{k+1}-\nabla_k \|^2\mid\mathcal{H}_k]\\
+2\bE[\langle W\my_k-\mathbf{1}\oy_k, \nabla_{k+1}-\nabla_k \rangle\mid \mathcal{H}_k]+2(n+\alpha Ln+1)\sigma^2.
\end{multline}
%where
%{\footnotesize\begin{multline}
%M_k:=
%2\alpha L \sum_{i=1}^n\bE[\|g_i(x_{i,k},\xi_{i,k})-\nabla f_i(x_{i,k})\|^2\mid \mathcal{H}_k]\\
%+\bE[\|G_{k+1}-\nabla_{k+1}\|^2\mid \mathcal{H}_k]\\+\bE[\|G_k-\nabla_k \|^2\mid \mathcal{H}_k]\\
%-2\sum_{i=1}^n\left(w_{ii}-\frac{1}{n}\right)\bE[\|g_i(x_{i,k},\xi_{i,k})-\nabla f_i(x_{i,k})\|^2\mid \mathcal{H}_k]\\
%= \bE[\|G_{k+1}-\nabla_{k+1}\|^2\mid \mathcal{H}_k]\\
%-\sum_{i=1}^n\left(2w_{ii}-1-\frac{2}{n}-2\alpha L\right)\bE[\|g_i(x_{i,k},\xi_{i,k})-\nabla f_i(x_{i,k})\|^2\mid \mathcal{H}_k].
%\end{multline}}

Now we bound $\|\nabla_{k+1}-\nabla_k \|^2$ and $\langle W\my_k-\mathbf{1}\oy_k, \nabla_{k+1}-\nabla_k \rangle$.
First, by Assumption \ref{asp: strconvexity} and Lemma \ref{lem: spectral norm},
\begin{multline*}
\|\nabla_{k+1}-\nabla_k \|^2\le L^2\|\mx_{k+1}-\mx_k\|^2=L^2\|W\mx_k-\mx_k-\alpha W\my_k\|^2=L^2\|(W-I)(\mx_k-\mathbf{1}\ox_k)-\alpha W\my_k\|^2\\
= \|W-I\|^2L^2\|\mx_k-\mathbf{1}\ox_k\|^2-2\alpha L^2\langle (W-I)(\mx_k-\mathbf{1}\ox_k), W\my_k\rangle+\alpha^2L^2\|W\my_k\|^2\\
=\|W-I\|^2L^2\|\mx_k-\mathbf{1}\ox_k\|^2-2\alpha L^2\langle (W-I)(\mx_k-\mathbf{1}\ox_k), W\my_k-\mathbf{1}\oy_k\rangle+\alpha^2L^2\|W\my_k-\mathbf{1}\oy_k\|^2+\alpha^2 nL^2\|\oy_k\|^2\\
\le \|W-I\|^2L^2\|\mx_k-\mathbf{1}\ox_k\|^2+2\alpha \|W-I\| L^2\rho_w\|\mx_k-\mathbf{1}\ox_k\|\|\my_k-\mathbf{1}\oy_k\|
+\alpha^2 L^2\rho_w^2\|\my_k-\mathbf{1}\oy_k\|^2+\alpha^2 nL^2\|\oy_k\|^2\\
\le 2\|W-I\|^2L^2\|\mx_k-\mathbf{1}\ox_k\|^2+2\alpha^2 L^2\rho_w^2\|\my_k-\mathbf{1}\oy_k\|^2+\alpha^2 nL^2\|\oy_k\|^2.
\end{multline*}
Second,
\begin{multline*}
\langle W\my_k-\mathbf{1}\oy_k, \nabla_{k+1}-\nabla_k \rangle\le L\rho_w\|\my_k-\mathbf{1}\oy_k\| \|(W-I)(\mx_k-\mathbf{1}\ox_k)-\alpha W\my_k\|\\
\le \|W-I\|L\rho_w\|\my_k-\mathbf{1}\oy_k\|\|\mx_k-\mathbf{1}\ox_k\|+\alpha L\rho_w\|\my_k-\mathbf{1}\oy_k\|\|W\my_k-\mathbf{1}\oy_k+\mathbf{1}\oy_k\|\\
\le \|W-I\|L\rho_w\|\my_k-\mathbf{1}\oy_k\|\|\mx_k-\mathbf{1}\ox_k\|+\alpha L\rho_w^2\|\my_k-\mathbf{1}\oy_k\|^2+\alpha \sqrt{n}L\rho_w\|\my_k-\mathbf{1}\oy_k\|\|\oy_k\|.
\end{multline*}
Notice that
\begin{equation*}
\|\oy_k\|\le \|\oy_k-h(\mx_k)\|+\|h(\mx_k)-\nabla F(\ox_k)\|+\|\nabla F(\ox_k)\|\le \|\oy_k-h(\mx_k)\|+\frac{L}{\sqrt{n}}\|\mx_k-\mathbf{1}\ox_k\|+L\|\ox_k-x^*\|.
\end{equation*}
We have 
\begin{multline*}
	\sqrt{n}L\rho_w\|\my_k-\mathbf{1}\oy_k\|\|\oy_k\|\le \sqrt{n}L\rho_w\|\my_k-\mathbf{1}\oy_k\|\left(\|\oy_k-h(\mx_k)\|+\frac{L}{\sqrt{n}}\|\mx_k-\mathbf{1}\ox_k\|+L\|\ox_k-x^*\|\right)\\
	\le L\rho_w^2\|\my_k-\mathbf{1}\oy_k\|^2+nL\|\oy_k-h(\mx_k)\|^2+ L^3\|\mx_k-\mathbf{1}\ox_k\|^2+\frac{1}{2}n L^3\|\ox_k-x^*\|^2,
	\end{multline*}
and
\begin{equation*}
\|\oy_k\|^2\le 3\|\oy_k-h(\mx_k)\|^2+\frac{3L^2}{n}\|\mx_k-\mathbf{1}\ox_k\|^2+3L^2\|\ox_k-x^*\|^2.
\end{equation*}
By (\ref{my_k+1-oy_k+1}) and the above relations,
\begin{multline}
\bE[\|\my_{k+1}-\mathbf{1}\oy_{k+1}\|^2\mid \mathcal{H}_k]
%\le \rho_w^2\bE[\|\my_{k}-\mathbf{1}\oy_{k}\|^2\mid \mathcal{H}_k]+\bE[\|\nabla_{k+1}-\nabla_k \|^2\mid\mathcal{H}_k]+2\bE[\langle W\my_k-\mathbf{1}\oy_k, \nabla_{k+1}-\nabla_k \rangle\mid \mathcal{H}_k]+M_k\\
\le \rho_w^2\bE[\|\my_{k}-\mathbf{1}\oy_{k}\|^2\mid \mathcal{H}_k]+2\|W-I\|^2L^2\|\mx_k-\mathbf{1}\ox_k\|^2\\
+2\alpha^2 L^2\rho_w^2\bE[\|\my_k-\mathbf{1}\oy_k\|^2\mid\mathcal{H}_k]+\alpha^2 nL^2\left(3\bE[\|\oy_k-h(\mx_k)\|^2\mid\mathcal{H}_k]+\frac{3L^2}{n}\|\mx_k-\mathbf{1}\ox_k\|^2+3L^2\|\ox_k-x^*\|^2\right)\\
+2\left(\|W-I\|L\rho_w\|\my_k-\mathbf{1}\oy_k\|\|\mx_k-\mathbf{1}\ox_k\|+\alpha L\rho_w^2\|\my_k-\mathbf{1}\oy_k\|^2\right)\\
+2\left(\alpha L\rho_w^2\bE[\|\my_k-\mathbf{1}\oy_k\|^2\mid\mathcal{H}_k]+\alpha nL\bE[\|\oy_k-h(\mx_k)\|^2\mid\mathcal{H}_k]+\alpha L^3\|\mx_k-\mathbf{1}\ox_k\|^2+\frac{1}{2}\alpha n L^3\|\ox_k-x^*\|^2\right)\\
+2(n+\alpha Ln+1)\sigma^2\\
\le \left(\rho_w^2+4\alpha L\rho_w^2+2\alpha^2 L^2\rho_w^2\right)\bE[\|\my_{k}-\mathbf{1}\oy_{k}\|^2\mid \mathcal{H}_k]\\
+\left(\beta\rho_w^2\bE[\|\my_k-\mathbf{1}\oy_k\|^2\mid\mathcal{H}_k]+\frac{1}{\beta}\|W-I\|^2 L^2\|\mx_k-\mathbf{1}\ox_k\|^2\right)\\
+\left(2\|W-I\|^2L^2+2\alpha L^3+3\alpha^2 L^4\right)\|\mx_k-\mathbf{1}\ox_k\|^2+\left(3\alpha^2 nL^4+\alpha nL^3\right)\|\ox_k-x^*\|^2\\
+\left[3\alpha^2L^2+2\alpha L+2(n+\alpha Ln+1)\right]\sigma^2\\
= \left(1+4\alpha L+2\alpha^2 L^2+\beta\right)\rho_w^2\bE[\|\my_{k}-\mathbf{1}\oy_{k}\|^2\mid \mathcal{H}_k]\\
+\left(\frac{1}{\beta}\|W-I\|^2 L^2+2\|W-I\|^2 L^2+3\alpha L^3\right)\|\mx_k-\mathbf{1}\ox_k\|^2+2\alpha nL^3\|\ox_k-x^*\|^2+M_{\sigma}
\end{multline}
for any $\beta>0$.
% In general, $\sigma^2+M_k\le 2n\sigma^2$. For sufficiently large $w_{ii}$ ($\sim 1/2$), $\sigma^2+M_k\le n\sigma^2$.

\addtolength{\textheight}{-12cm}   % This command serves to balance the column lengths

\subsection{Proof of Lemma \ref{lem: rho_M}}
\label{subsec: proof lemma rho_M}
The characteristic function of $M$ is given by
\begin{multline}
g(\lambda):=\text{det}(\lambda I-M)=(\lambda-m_{11})(\lambda-m_{22})(\lambda-m_{33})
-a_{23}a_{32}(\lambda-m_{11})-a_{13}a_{31}(\lambda-m_{22})\\
-a_{12}a_{21}(\lambda-m_{33})-a_{12}a_{23}a_{31}-a_{13}a_{32}a_{21}.
\end{multline}
Necessity is trivial since $\text{det}(\lambda^* I-M)\le 0$ implies $g(\lambda)=0$ for some $\lambda\ge \lambda^*$. We now show $\text{det}(\lambda^* I-M)>0$ is also a sufficient condition.

Given that $g(\lambda^*)=\text{det}(\lambda^* I-M)>0$,
\begin{equation*}
(\lambda^*-m_{11})(\lambda^*-m_{22})(\lambda^*-m_{33})> a_{23}a_{32}(\lambda^*-m_{11})+a_{13}a_{31}(\lambda^*-m_{22})+a_{12}a_{21}(\lambda^*-m_{33}).
\end{equation*}
It follows that
\begin{equation}
\begin{array}{ccc}
\gamma_1(\lambda^*-m_{22})(\lambda^*-m_{33}) & > & a_{23}a_{32}\\
\gamma_2(\lambda^*-m_{11})(\lambda^*-m_{33}) & > & a_{13}a_{31}\\
\gamma_3(\lambda^*-m_{11})(\lambda^*-m_{22}) & > & a_{12}a_{21}
\end{array}
\end{equation}
for some $\gamma_1,\gamma_2,\gamma_3>0$ with $\gamma_1+\gamma_2+\gamma_3\le 1$.
Consider
\begin{multline*}
g'(\lambda)=(\lambda-m_{22})(\lambda-m_{33})+(\lambda-m_{11})(\lambda-m_{33})+(\lambda-m_{11})(\lambda-m_{22})-a_{23}a_{32}-a_{13}a_{31}\\
-a_{12}a_{21}.
\end{multline*}
We have $g'(\lambda)>0$ for $\lambda\in(-\infty,-\lambda^*]\cup[\lambda^*,+\infty)$. Noticing that
\begin{multline*}
g(-\lambda^*)\le -(\lambda^*+m_{11})(\lambda^*+m_{22})(\lambda^*+m_{33})+a_{23}a_{32}(1+m_{11})+a_{13}a_{31}(\lambda^*+m_{22})\\
+a_{12}a_{21}(\lambda^*+m_{33})< 0,
\end{multline*}
all real roots of $g(\lambda)=0$ lie in the interval $(-\lambda^*,\lambda^*)$. By the Perron-Frobenius theorem, $\rho_M$ is an eigenvalue of $M$. We conclude that $\rho_M<\lambda^*$.
\end{document}